\documentclass[10pt]{article}
\usepackage[top=4cm,bottom=3cm,left=5cm,right=5cm]{geometry}               
\usepackage{amsmath,amscd}
\usepackage[english]{babel}
\usepackage{amsthm}
\usepackage{fancyhdr}
\usepackage{latexsym}
\usepackage{amssymb}
\usepackage{color}
\usepackage{makeidx}
\usepackage{graphicx}
\usepackage{amssymb}
\usepackage{mathtools}
\usepackage{epstopdf}
\usepackage[mathscr]{eucal}
\usepackage{stackrel}
\usepackage{enumerate} %per avere gli elenchi gaggi: ad esempio nel dimostrare i "sse"
\usepackage{setspace}
\usepackage{comment}
\usepackage{mathrsfs}

\DeclareMathOperator*{\End}{End}

\def\vF{\mathbb{F}}

\def\vN{\mathbb{N}}

\newtheorem{teo}{Theorem}
\newtheorem{lemma}[teo]{Lemma}
\newtheorem{cor}[teo]{Corollary}

\theoremstyle{definition}
\newtheorem{dfn}[teo]{Definition}

\newtheorem{example}[teo]{Example}
\newtheorem{rem}[teo]{Remark}

\makeindex

\author{Giacomo Micheli, Davide Schipani\\
Institut f\"ur Mathematik\\
Universit\"at Z\"urich}
\title{On Canonical Subfield Preserving Polynomials}
\begin{document}

\maketitle
%\clearpage{\pagestyle{empty}\cleardoublepage}

                        %crea l'indice

\rhead[\fancyplain{}{\bfseries\leftmark}]{\fancyplain{}{\bfseries\thepage}}
\lhead[\fancyplain{}{\bfseries\thepage}]{\fancyplain{}{\bfseries INDICE}}

%%%%%%%%%%%%%%%%%%%%%%%%%%%%%%%%%%%%%%%%%%%%%%%%%%%%%%%%%%%%%%%%%%%%%%%%%%%%%%%%

\lhead[\fancyplain{}{\bfseries\thepage}]{\fancyplain{}{\bfseries\rightmark}}

\begin{abstract}
Explicit monoid structure is provided for the class of canonical subfield preserving polynomials over finite fields. 
Some classical results and asymptotic estimates will follow as corollaries.
\end{abstract}

\pagenumbering{arabic}
\section{Introduction}
Let $q$ be a prime power and $m$ a natural number. In \cite{carlitz} the structure of the
group consisting of permutation polynomials \cite{lidl} of $\vF_{q^m}$ having coefficients in the base field $\vF_q$ was made explicit.
We start  observing that, if $f$ is a permutation of $\vF_{q^m}$ with coefficients in $\vF_{q}$ then
$$f(\vF_q)=\vF_q\quad\text{and}\quad
\forall\ d,s\mid m\ \ \ \ \ \ f(\vF_{q^d}\setminus\vF_{q^s})=\vF_{q^d}\setminus\vF_{q^s}.
$$
Indeed for any integer  $s\geq 1$, since $f$ has coefficients in $\vF_{q}$ and $\vF_{q^s}$ is a field, we have $f(\vF_{q^s})\subseteq  \vF_{q^s}$. Being $f$ also a bijection, this is also an equality. The property above follows then 
directly (see also \cite[Lemma 2]{carlitz}).

It is natural now to ask which are the polynomials $f$, having coefficients in $\vF_{q}$, such that
\begin{equation}\label{preserving}
f(\vF_q)\subseteq\vF_q\quad\text{and}\quad
\forall\ d,s\mid m\ \ \ \ \ f(\vF_{q^d}\setminus\vF_{q^s})\subseteq\vF_{q^d}\setminus\vF_{q^s}.
\end{equation}
Let us call $T_q^m$ the set of such %\emph{subfield preserving} 
polynomials. We remark that this is a monoid under composition and its invertible elements $(T_q^m)^*$  consist of the group of permutation polynomials with coefficients in $\vF_{q}$ mentioned above.
In this paper we give the explicit semigroup structure of $T_q^m$, obtaining the main result of \cite{carlitz} (i.e. the group structure mentioned above) as a corollary.
%To that end  an indexing of the finite field based on the cycles of the Frobenius morphism $\varphi_q$ will be fundamental.
The explicit semigroup structure will allow us to compute the probability that a 
polynomial chosen uniformly at random having coefficients in $\vF_{q}$ satisfies condition (\ref{preserving}). %is subfield preserving.
%As additional results we will prove that
This will imply the following remarkable results:
\begin{itemize}
 \item{Given $p$ prime, for $q$ relatively large, the density of $T_q^p$ is approximately zero.}
 \item{Given $q$, for $p$ relatively large prime, the density of $T_q^p$ is approximately one.}
\item{For $q=p$ large prime %and $m=p$
 the density of  $T_p^p$ is  approximately $1/e$.}
\end{itemize} 
\section{Preliminary definitions}\label{PD}
%Define
\begin{dfn}
We say $f: \vF_{q^m}\rightarrow\vF_{q^m}$  to be \emph{subfield preserving} if 
\begin{equation}\label{property}
f(\vF_q)\subseteq\vF_q\quad\text{and}\quad\forall\ d,s\mid m\ \ \ \ \ \ f(\vF_{q^d}\setminus\vF_{q^s})\subseteq\vF_{q^d}\setminus\vF_{q^s}.
\end{equation}
Moreover, we will say $f$ to be $q$-\emph{canonical} if its polynomial representation has coefficients in $\vF_q$ (or simply \emph{canonical} when $q$ is understood).
\end{dfn}

\begin{rem}
One of the reason why we use the term \emph{canonical} to address the property of having coefficients in a subfield is that, under this property,
 the induced application $\tilde f$ of $f(x)$ is always well defined no matter what irreducible polynomial we choose for the representation of the finite field extension $\vF_{q^m}$.
\end{rem}
Denote by
$\mathcal{L}_{\vF_{q^m}}$
%\mathrel{\mathop:}=
%\{ f: \vF_{q^m}\to \vF_{q^m} \mid  \quad   \forall\ d,s\mid m\ \quad %f(\vF_{q^d}\setminus\vF_{q^s})\subseteq\vF_{q^d}\setminus\vF_{q^s}  \}
%\]
the set of all subfield preserving polynomials.
\begin{rem}
If we drop the condition on the coefficients, the semigroup structure becomes straightforward: 

$$\mathcal{L}_{\vF_{q^m}} \cong\  \stackrel[k|m]{}{\bigtimes}M_{[k\pi(k)]}.$$
with $\pi(k)$ being the number of monic irreducible polynomials of degree $k$ over $\vF_q$ and $M_{[n]}$ being the set of all maps from $\{1,\ldots,n\}$ to itself.
\end{rem}
\begin{rem}
Clearly not all subfield preserving polynomials are canonical, which can also be checked by a cardinality count with the results later in the paper.

\end{rem}

In the rest of the paper we will need the following lemma, whose proof can be easily adapted from \cite{carlitz} and \cite{chen}.
\begin{lemma}
Let $f: \vF_{q^m}\rightarrow\vF_{q^m}$ be a map. Then $f\in\vF_{q}[x]$ if and only if $f\circ \varphi_q=\varphi_q\circ f$ where $\varphi_q(x)=x^q$.
\end{lemma}
Indeed the set of functions we are looking at consists of $T_q^m=\mathcal{L}_{\vF_{q^m}}\cap\ \mathcal{C}_{\varphi_q}$ where $ \mathcal{C}_{\varphi_q}\mathrel{\mathop:}=
\{ f: \vF_{q^m}\to \vF_{q^m} \mid   f\circ \varphi_q=\varphi_q\circ f  \}
$.

\section{Combinatorial underpinning}\label{CU}

Let $S$ be a finite set and $\psi: S \to S$ a bijection.
For any $T\subseteq S$, let
$$
\mathcal{K}_{\psi}(T)\mathrel{\mathop:}=
\{ f: T\to T \mid   \forall x\in T\ \  f\circ \psi(x)=\psi\circ f(x) \}.
$$
For any partition $\mathcal{P}$ of $S$ into sets $P_k$, let 
$$
M_{S}(\mathcal{P})\mathrel{\mathop:}=
\{ f: S\to S \mid   \forall\ k\ \  f(P_k)\subseteq P_k  \}.
$$
When $\mathcal P=\{S\}$ is the trivial partition, we will denote $M_{S}(\{S\})=M_S$ namely the monoid of applications
from $S$ to itself.

For any bijection $\phi: S \to S$, define $\phi_k$ for any $k$ as the composition of the cycles of $\phi$ of length $k$, and set $\phi_k=(\emptyset)$ if $\phi$ has no cycles of length $k$. Let $W$ denote the set $\{1,\ldots,|S|\}$, then $\phi=\stackrel[k\in W,\phi_k\neq(\emptyset)]{}{\prod} \phi_k$. If $\text{supp}(\phi_k)$ denotes the set of elements moved by $\phi_k$, then $\phi$ induces a partition $\mathcal{P}_{\phi}$ on $S=\stackrel[k\in W]{}{\bigcup} 
S_k$, with $S_k=\text{supp}(\phi_k)$, for $k\geq 2$, and $S_1$ being the set of fixed points of $\phi$.

%Let $S_1$ be the set of fixed points of $\phi$ and set $S_k=\text{supp}(\phi_k)$ for $k\geq 2$.

\begin{lemma}\label{lemma1}
$$M_{S}(\mathcal{P}_{\phi})\cap \mathcal{K}_{\phi}(S) \cong \stackrel[k\in W,\phi_k\neq(\emptyset)]{}{\bigtimes}\mathcal{K}_{\phi_k}(S_k)$$
\end{lemma}
\begin{proof}
Clearly any $f\in \mathcal{K}_{\phi_k}(S_k)$ can be extended to $S$ as the identity and then the extension $\bar f$ belongs to $\mathcal{K}_{\phi}(S)\cap M_{S}(\mathcal{P}_{\phi})$. Indeed we have a natural injection
$$
\stackrel[k\in W,\phi_k\neq(\emptyset)]{}{\bigtimes}\mathcal{K}_{\phi_k}(S_k)\hookrightarrow M_{S}(\mathcal{P}_{\phi})\cap \mathcal{K}_{\phi}(S).
$$
This is also a surjection: in fact let $f\in M_{S}(\mathcal{P}_{\phi})\cap \mathcal{K}_{\phi}(S)$ and define
$$
f_k(x)\mathrel{\mathop:}=\begin{cases} f(x) &\mbox{if } x\in S_k, \\
x & \mbox{otherwise}.    \end{cases}
$$
Since $M_{S}(\mathcal{P}_{\phi})\cap \mathcal{K}_{\phi}(S)\subseteq M_{S}(\mathcal{P}_{\phi})$, then $f_k(S_k)\subseteq S_k$ which implies 
$$f_k\big|_{S_k} \in \mathcal{K}_{\phi_k}(S_k).$$
As the $S_k$ form a partition, the composition of all the $f_k$ coincides with $f$.
\end{proof}
Now, for $n,k\in \vN$ let $U_n^k$ be a set with $kn$ elements and $\psi$ a bijection of $U_n^k$ having $n$ cycles of length $k$. Let us put indeces on the elements of the set in the following way: let $a_{ij}$ be the $j$-th element of the $i$-th cycle, with $i\in\{1,\ldots,n\}$ and $j\in\{1,\ldots,k\}$.

Let $[h]$ denote $\{1,\ldots,h\}$ for a natural number $h$. We say $\lambda: [h]\to [h]$ to be a cyclic shift of $[h]$ if  $\lambda(j+\ell)=\lambda(j)+\ell$ modulo $h$ for any $j,\ell\in[h]$.

Let $\gamma_1,\ldots,\gamma_n$ be cyclic shifts of $[k]$ and $\sigma: [n]\to [n]$ a map. We construct then $f_{\sigma}^{\gamma}: U_n^k\to U_n^k$ as follows:
$$
f_{\sigma}^{\gamma}(a_{ij})\mathrel{\mathop:}=a_{\sigma(i)\gamma_i(j)}.
$$
\begin{teo}\label{teo2}
$
g\in \mathcal{K}_{\psi}(U_n^k) \iff \exists\ \gamma\mathrel{\mathop:}=(\gamma_1,\ldots,\gamma_n)$, $\gamma_i$ cyclic shifts of $[k]$, and $\exists\ \sigma: [n]\to [n]$ map such that $g=f_{\sigma}^{\gamma}$.
\end{teo}
\begin{proof}
Suppose first $g\in \mathcal{K}_{\psi}(U_n^k)$. Then 
$$g(a_{ij})=g(\psi^{j-1}(a_{i1}))=\psi^{j-1}(g(a_{i1})).$$
Define $\sigma(i)\mathrel{\mathop:}=[g(a_{i1})]_1$ and $\gamma_i(j)\mathrel{\mathop:}=[g(a_{ij})]_2$, where the subscripts $[x]_1$ and $[x]_2$ refer to the two indeces of $x\in U_n^k$ in the representation $a_{ij}$ above.

Observe that for all $i\in[n]$, $\gamma_i$ is a cyclic shift, indeed it holds modulo $k$:
$$
\gamma_i(j+\ell)=[g(a_{i\ j+\ell})]_2=[g(\psi^\ell(a_{ij}))]_2=$$$$[\psi^\ell(g(a_{ij}))]_2=[g(a_{ij})]_2+\ell=\gamma_i(j)+\ell .
$$
Moreover remark that
$$
g(a_{ij})=g(\psi^{j-1}(a_{i1}))=\psi^{j-1}(g(a_{i1}))=\psi^{j-1}(a_{\sigma(i)\gamma_i(1)})=
$$
$$ a_{\sigma(i)\ \gamma_i(1)+j-1}=a_{\sigma(i)\gamma_i(j)}=f_{\sigma}^{\gamma}(a_{ij}).
$$
Let us prove now the other implication:
$$
\psi(f_{\sigma}^{\gamma}(a_{ij}))=\psi(a_{\sigma(i)\gamma_i(j)})=a_{\sigma(i)\ \gamma_i(j)+1}=
$$
$$a_{\sigma(i)\gamma_i(j+1)}=f_{\sigma}^{\gamma}(a_{i\ j+1})=f_{\sigma}^{\gamma}(\psi(a_{ij}))
$$
for all $i\in [n]$ and $j\in [k]$.
\end{proof}

\subsection{Semidirect product of monoids}
We now recall the definition of semidirect product of monoids
\begin{dfn}
Let $M,N$ be monoids and let $\Gamma: M \to \End(N)$ with $m\mapsto \Gamma_m$ be an antihomomorphism of monoids (i.e. $\Gamma_{m_1 m_2}=\Gamma_{m_2} \circ \Gamma_{m_1}$).
We define $M\ltimes_\Gamma N$ as the monoid having support $M\times N$ and operation $*$ defined by the formula
\[(m_1,n_1)*(m_2,n_2)=(m_1m_2, \Gamma_{m_2}(n_1) \,n_2)\]
\end{dfn}
\begin{rem}
It is straightforward to verify that the associative property holds.
\end{rem}
We will now prove an easy lemma that will be useful in Section \ref{SS}.
For any monoid $H$ let us denote by $H^*$ the group of invertible elements of $H$.
\begin{lemma}\label{prodmonprop}
Let $M\ltimes G$ be a semidirect product of monoids where $G$ is a group. Then
\[(M\ltimes G)^*= M^*\ltimes G\]
\end{lemma}
\begin{proof}
The inclusion  $(M\ltimes G)^*\subseteq M^*\ltimes G$ is trivial, since if $(m,g)\in (M\ltimes G)^*$ then there exists $(m',g')$ such that
\[(m,g)*(m',g')=(e_1,e_2)\]
so $m m'=e_1$ identity element of $M$.
Let us now prove $(M\ltimes G)^*\supseteq M^*\ltimes G$. Let $(m,g)\in M^*\ltimes G$, then its inverse is $(m^{-1},\Gamma_{m^{-1}}(g^{-1}))$.
\end{proof}

We are now ready to prove the main proposition of this section as a corollary of  Theorem \ref{teo2}.

We first observe that the set of cyclic shifts of $[k]$ is clearly isomorphic to $C_k$,  the cyclic group of order $k$, and each cyclic shift can be identified by its action on $1$. 
\begin{cor}\label{main}
\[\mathcal{K}_{\psi}(U^k_n)\cong M_{[n]}\ltimes_\Gamma C_k^n\]
where $\Gamma$ si defined by
\[\Gamma(\sigma)(\gamma):=\Gamma_\sigma(\gamma):=\gamma_\sigma:=(\gamma_{\sigma(1)},\dots, \gamma_{\sigma(n)})\]
for any $\gamma\in C_k^n$.
\end{cor}
\begin{proof}

The reader should first observe \[\Gamma_{\mu}(\gamma_{\sigma(1)},\dots, \gamma_{\sigma(n)})=(\gamma_{\sigma(\mu(1))},\dots,\gamma_{\sigma (\mu(i))},\dots,\gamma_{\sigma(\mu(n))})\]
for any $\sigma, \mu\in M_{[n]}$. This can be easily seen by denoting $\gamma_{\sigma(i)}=:g_i$.
Therefore, $\Gamma$ is an antihomomorphism, as we wanted:
\[\Gamma(\sigma \mu)(\gamma)=\gamma_{\sigma \mu}=(\gamma_{\sigma(\mu(1))},\dots,\gamma_{\sigma (\mu(i))},\dots,\gamma_{\sigma(\mu(n))})=\]
\[\Gamma_{\mu}(\gamma_{\sigma(1)},\dots, \gamma_{\sigma(n)})=\Gamma_{\mu}\circ \Gamma_{\sigma}(\gamma).\]
Let
\[\Delta: M_{[n]}\ltimes C_k^n \longrightarrow \mathcal{K}_{\psi}(U^k_n)\]
\[(\sigma, \gamma)\mapsto f^\gamma_\sigma.\]
$\Delta$ is clearly a bijection by Theorem \ref{teo2}. It is also an automorphism since
\[\Delta((\overline \sigma, \overline \gamma)*( \sigma,\gamma))(a_{i,j})=\Delta(\overline \sigma \sigma,\overline \gamma_\sigma \gamma)(a_{i,j})=f_{\overline \sigma \sigma}^{\overline \gamma_\sigma \gamma}(a_{i,j})=\]
\[a_{\overline \sigma \sigma(i),\overline \gamma_{\sigma(i)} \gamma_{i} (j)}=f_{\overline{\sigma}}^{\overline{\gamma}}(a_{\sigma(i), \gamma_i(j)})=f_{\overline{\sigma}}^{\overline{\gamma}}\circ f_{\sigma}^{\gamma}(a_{i,j})=\]
\[\bigl(\Delta(\overline \sigma, \overline \gamma)\circ\Delta( \sigma,\gamma)\bigr)(a_{i,j})\]
for all $i\in [n]$ and all $j\in [k]$.
\end{proof}

\section{Semigroup structure of $T_q^m$}\label{SS}
Consider now $T_q^m$ and notice that, since $M_{\vF_{q^m}}(\mathcal{P}_{\varphi_q})=\mathcal{L}_{\vF_{q^m}}$ and $\mathcal{K}_{\varphi_q}(\vF_{q^m})=\mathcal{C}_{\varphi_q}$, then we have 
\begin{equation}\label{eqprin}
T_q^m=\mathcal{L}_{\vF_{q^m}}\cap\mathcal{C}_{\varphi_q}=M_{\vF_{q^m}}(\mathcal{P}_{\varphi_q})\cap \mathcal{K}_{\varphi_q}(\vF_{q^m}).
\end{equation}
Indeed the condition 
\[f(S_k)\subseteq S_k\]
for each $S_k$ in the partition induced by $\varphi_q$ is equivalent to the subfield preserving requirement (\ref{property}), being 
\[S_1=\vF_q\quad \text{and}\quad S_k=\bigcap_{a|k,\,a\neq k}\left(\vF_{q^k}\setminus \vF_{q^a}\right)\quad \text{for $k\geq2$}.\]
Any element $\alpha$ in a cycle of length $d$ is associated to the irreducible polynomial $\stackrel[i=0]{d-1}{\prod}(x-\alpha^{q^i})\in\vF_{q}[x] $, so there is a bijection between the cycles of $\varphi_q$ of length $d$
and the monic irreducible polynomials of degree $d$
over $\vF_q$, whose cardinality is
\[\pi(d)={1\over d}\sum_{j\,| d}\mu(d/j) q^j\]
with $\mu$ being the Moebius function.
Now, write 
\[\varphi_q=\prod_{k\,|\,m} \phi_k\]
similarly as above with $\phi=\varphi_q$
 and label the elements of the finite field as follow:
$a_{i,j}^{(k)}$ is the $j$-th element living in the $i$-th $k$-cycle.
\begin{example}
Let $\vF_{2^2}=\vF_2[\alpha]/{(\alpha^2+ \alpha +1)}$ consisting of $\{0,1,\alpha,\alpha+1\}$. Indeed
\[\varphi_q= \phi_1 \phi_2= (0)(1)(\alpha,\alpha+1)\]
and then  $a_{1,1}^{(1)}=0$, $a_{2,1}^{(1)}=1$, $a_{1,1}^{(2)}=\alpha$ and $a_{1,2}^{(2)}=\alpha+1$.
\end{example}

\begin{teo}
\begin{equation}\label{maintheorem}
T_q^m\cong  \stackrel[k|m]{}{\bigtimes}M_{[\pi(k)]} \ltimes C_k^{\pi(k)}
 \end{equation}
 %where $C_k$ is the cyclic group of order $k$.
\end{teo}
\begin{proof}
It follows from Lemma \ref{lemma1} and %Theorem \ref{teo2} 
Corollary \ref{main} using the partition induced by the Frobenius morphism.
Indeed,
using equation \ref{eqprin} and Lemma \ref{lemma1} we get
\[T_q^m\cong\ \stackrel[k\in W,\phi_k\neq(\emptyset)]{}{\bigtimes}\ \mathcal{K}_{\phi_k}(S_k).\]
Using now Corollary \ref{main} we get 
\[T_q^m\cong  \stackrel[k|m]{}{\bigtimes}M_{[\pi(k)]} \ltimes C_k^{\pi(k)}.\]
More explicitely, the action of $t\in \stackrel[k|m]{}{\bigtimes}M_{[\pi(k)]} \ltimes C_k^{\pi(k)}$ on an element
$a^{(k)}_{i,j}\in S_k\subseteq \vF_{q^m}$ is given by 
\[t(a^{(k)}_{i,j})=(\sigma^{(k)}, \gamma^{(k)})(a^{(k)}_{i,j})=f^{\gamma^{(k)}}_{\sigma^{(k)}}(a^{(k)}_{i,j})= a^{(k)}_{{\sigma^{(k)}(i)},{\gamma_i^{(k)}(j)}}\]
where $\gamma^{(k)}$ and $\sigma^{(k)}$ are the components indexed by $k$.

% The homomorphism connected to the semidirect product of semigroups sends $\sigma\in M_{[\pi(k)]}$ in $h_{\sigma}:C_k^{\pi(k)}\to C_k^{\pi(k)}\in \End(C_k^{\pi(k)})$, where $$h_{\sigma}((\gamma_1,\ldots,\gamma_{\pi(k)}))\mathrel{\mathop:}=(\gamma_{\sigma(1)},\ldots,\gamma_{\sigma(\pi(k))})$$ and the product is given, as expected, by the following formula for each $k$:
%$$
%f_{\sigma}^{\gamma}\circ f_{\bar\sigma}^{\bar\gamma}(a^{(k)}_{ij}) %= f_{\sigma}^{\gamma}( f_{\bar\sigma}^{\bar\gamma}(a^{(k)}_{ij}))=f_{\sigma}^{\gamma}(a^{(k)}_{\bar\sigma(i)\bar\gamma_i(j)})=
%=a^{(k)}_{\sigma(\bar\sigma(i))\gamma_{\bar\sigma(i)}(\bar\gamma_i(j))}
%$$
%where clearly $\sigma=\sigma^{(k)}$ and $\gamma=\gamma^{(k)}$.
\end{proof} 
\begin{cor}\label{carli}
$$ (T_q^m)^*\cong\ \stackrel[k|m]{}{\bigtimes} \mathcal{S}_{\pi(k)}\ltimes C_k^{\pi(k)}  $$
where $\mathcal{S}_{\pi(k)}$ is the permutation group of $\pi(k)$ elements.
\end{cor}
\begin{proof}
Observe that  
$$ (T_q^m)^*\cong\ \stackrel[k|m]{}{\bigtimes} (M_{[\pi(k)]}\ltimes C_k^{\pi(k)} )^* $$
holds trivially.
Applying now Lemma \ref{prodmonprop} %to equation \ref{maintheorem} getting
yields
$$ (T_q^m)^*\cong\ \stackrel[k|m]{}{\bigtimes} (M_{[\pi(k)]}\ltimes C_k^{\pi(k)} )^* \cong\ \stackrel[k|m]{}{\bigtimes} \mathcal{S}_{\pi(k)}\ltimes C_k^{\pi(k)}.$$
\end{proof}
\begin{cor}\label{croc}
\[|T_q^m|=\prod_{k\,| m} k^{\pi(k)}\pi(k)^{\pi(k)}\]
\[|(T_q^m)^*|=\prod_{k\,| m} k^{\pi(k)}\pi(k)!\]
\end{cor}
\begin{rem}
Corollary \ref{carli} corresponds to \cite[Theorem 2]{carlitz} and Corollary \ref{croc} generalizes the corollary of \cite[Theorem 2]{carlitz}.
\end{rem}

\begin{rem}
Let us observe that a simpler construction as a direct product for $(T_q^m)^*$ can also be seen as follows:
\begin{itemize}
\item First notice that any permutation polynomial over $\vF_q$  can be extended to a permutation polynomial over $\vF_{q^m}$ with coefficients in $\vF_q$ by simply defining it as the identity function on $\vF_{q^m}\setminus \vF_q$ and Lagrange interpolating over the whole field. The produced  permutation polynomial over $\vF_{q^m}$ has coefficients in $\vF_q$,
since it commutes with $\varphi_q$, which is easily checked by looking at the base field and the rest separately.
\item 
{$ (T_q^m)^*$  has then a normal subgroup isomorphic to $\mathcal{S}_q$ consisting of
 \[\{s\in {T_q^m}^*\,|\, s  \ \text{is the identity on $\vF_{q^m}\setminus \vF_q$} \}.\]}
\item
Let \[H^m_q:= \{h \in (T_q^m)^*\, |\, h \ \text{is the identity on $\vF_q$}\}.\]
$H^m_q$ is also normal in $(T_q^m)^*$.
\item
$\mathcal{S}_q\times H^m_q  = (T_q^m)^*$. Indeed note first that $H^m_q\cap \mathcal{S}_q=1$. Now
given $f\in (T_q^m)^*$ we have to prove that it can be written as a composition of an element of
$H^m_q$ and an element of $\mathcal{S}_q$. Let $s_2\in \mathcal{S}_q$ such that $s_2$ restricted to $\vF_q$ is
$f$. Let $s_1(x)\in \mathcal{S}_q$ such that $s_1$ restricted to $\vF_q$ is the inverse permutation of the restriction
of $f$ to $\vF_q$. In other words $f(s_1(x))$ restricted to $\vF_q$ is the identity.
Observe then that, since $f(s_1(x))$ has also coefficients in $\vF_q$, it lives in $H^m_q$.
Verify that $s_2(f(s_1(x)))=f$. And so we have written $f$ as a composition of an element of
$\mathcal{S}_q$ and an element of $H^m_q$.
\end{itemize}
\end{rem}

\section{Asymptotic density of $T_q^m$}\label{AD}

Let us first  compute the asymptotic density of the group of permutation polynomials described in \cite{carlitz} inside the whole group
of permutation polynomials, and inside the monoid of the polynomial functions having coefficients in the subfield $\vF_{q}$.  We will restrict to the case $\vF_{q^p}$, $p$ prime.
\begin{teo}\label{ciao}
Consider an element of   $ \vF_{q}[x]/(x^{q^p}-x)$  chosen  uniformly  at random. The probability that this is a permutation polynomial tends to $0$ as %at least one of 
$p$ and/or $q$ tends to $\infty$.
\end{teo}
\begin{proof}
Given Corollary \ref{croc}, we need to consider 
$$
L\mathrel{\mathop:}=\stackrel[p\lor q \to \infty]{}{\lim}\frac{q! (p)^{\frac{q^p-q}{p}}(\frac{q^p-q}{p})!}{q^{q^p}}.
$$
By Stirling approximation this is
$$
L=\stackrel[p\lor q \to \infty]{}{\lim}\frac{q! (p)^{\frac{q^p-q}{p}}(\frac{q^p-q}{pe})^{\frac{q^p-q}{p}}\sqrt{2\pi\frac{q^p-q}{p}}}{q^{q^p}}.
$$
Now notice that
$$
\stackrel[p\lor q \to \infty]{}{\lim} \left(\frac{q^p-q}{q^p}\right)^{\frac{q^p-q}{p}}=\stackrel[p\lor q \to \infty]{}{\lim} \left(1-\frac{1}{q^{p-1}}\right)^{q^{p-1}\cdot\frac{q-q^{2-p}}{p}}%=\stackrel[p\lor q \to \infty]{}{\lim} \left(1-\frac{1}{q^{p-1}}\right)^{q^{p-1}\cdot\frac{q}{p}}
$$

By the continuity of the exponential function, this can be written as %is equal to
$$
\stackrel[p\lor q \to \infty]{}{\lim} e^{\frac{q-q^{2-p}}{p}\ln \left(1-\frac{1}{q^{p-1}}\right)^{q^{p-1}} }=e^{-\stackrel[p\lor q \to \infty]{}{\lim} \frac{q}{p}}
$$
so that
$$
L=\stackrel[p\lor q \to \infty]{}{\lim}\frac{q!(q^p)^{\frac{q^p-q}{p}}e^{- \frac{q}{p}}\sqrt{2\pi\frac{q^p-q}{p}}}{q^{q^p}e^{\frac{q^p-q}{p}}}.
$$
$$
=\stackrel[p\lor q \to \infty]{}{\lim}\frac{q!e^{- \frac{q}{p}}\sqrt{2\pi\frac{q^p-q}{p}}}{q^{q}e^{\frac{q^p-q}{p}}}=0,
$$
as one can easily see by exploring the cases $q\to \infty$ with Stirling and $q$ fixed.
\end{proof}
By observing that $q^p!>q^{q^p}$ definitively for large $p$ and/or $q$, we have also the following:
\begin{cor}
Consider a permutation of the set  $ \vF_{q^m}$  chosen  uniformly  at random. 
The probability that its associated permutation polynomial has coefficients in the subfield $\vF_q$ tends to $0$ as $p$ and/or $q$ tends to $\infty$.
\end{cor}
We are now interested in an asymptotic estimate for the density of 
$T_q^p$ in $\vF_q[x]/(x^{q^p}-x)$ for $p$ prime number. We will show in fact that the monoid of canonical subfield preserving polynomials has nontrivial density inside the monoid of polynomial functions having coefficients in the subfield $\vF_q$.
%\begin{rem}
Given Corollary \ref{croc}, the probability that an element of   $ \vF_{q}[x]/(x^{q^p}-x)$  chosen uniformly at random is subfield preserving is
\[{|T_q^p|\over q^{q^p}}=\frac{q^q (q^p-q)^{\frac{q^p-q}{p}}}{q^{q^p}}.\]
%\end{rem}
\begin{teo}\label{expo}
Consider an element  of   $ \vF_{q}[x]/(x^{q^p}-x)$  chosen  uniformly  at random. The probability that this is subfield preserving tends to $e^{- \stackrel[p\lor q \to \infty]{}{\lim}\frac{q}{p}}$ as $p$ and/or $q$ tends to $\infty$.
\end{teo}
\begin{proof}
We need to consider 
$$
\ell\mathrel{\mathop:}=\stackrel[p\lor q \to \infty]{}{\lim}\frac{q^q (q^p-q)^{\frac{q^p-q}{p}}}{q^{q^p}}.
$$
With similar arguments as in Theorem \ref{ciao}, this transforms to

$$
\ell=\stackrel[p\lor q \to \infty]{}{\lim} \frac{q^q (q^p)^{\frac{q^p-q}{p}}}{q^{q^p}}e^{- \frac{q}{p}}=e^{-\stackrel[p\lor q \to \infty]{}{\lim} \frac{q}{p}}
$$

\end{proof}

\begin{cor}
$\,$
\begin{itemize}
\item $\stackrel[p\to \infty]{}\lim{\frac{|T_q^p|}{ q^{q^p}}}=1$, if $q$ is fixed.
\item $\stackrel[q\to \infty]{}\lim{\frac{|T_q^p|}{ q^{q^p}}}=0$, if $p$ is fixed.
\end{itemize}
\end{cor}

\begin{cor}\label{interestcor}
Let $q=p$.
\[\lim_{p\to \infty}{\frac{|T_p^p|}{ p^{p^p}}}=1/e\]
\end{cor}

\begin{rem}
Clearly all the limits above are computed for $p$ and $q$ running over the natural numbers, but they hold in particular for the subsequences of increasing primes $p$ and possible orders of finite fields $q$. 
\end{rem}

\section{Example}\label{EE}
%Let us now consider the case $p=m=2$.
Let us consider the structure of $T_2^2$ as an example. Let $\alpha$ be a root of $x^2+x+1=0$, so that $\vF_{2^2}=\vF_{2}[\alpha]/(\alpha^2+\alpha+1)$. It is easy to check that for each polynomial $f\in L$ with $$L\mathrel{\mathop:}=\{0,1,x^2+x,x^2+x+1,x^3,x^3+1,x^3+x^2+x,x^3+x^2+x+1 \}$$ we have $f(\alpha)\in \vF_{2}$.
We know that $T_2^2$ contains $8$ polynomials, so that 
$$T_2^2=\frac{\vF_{2}[x]}{(x^4-x)}\Large\setminus L =$$
$$ \{x,x+1,x^2,x^2+1,x^3+x^2+1,x^3+x,x^3+x^2,x^3+x+1 \}.$$

The structure is $C_2 \bigtimes M_2$. 

Indeed $C_1^{2} \rtimes M_{2}=M_2$ and consists of  $$\{x,x^2+1,x^3+x^2,x^3+x+1 \}$$ that is those functions which fix $\vF_4\setminus\vF_2$ and act as $ M_{2}$ on $\vF_2$.

Also $C_2 \rtimes M_{1}=C_2$ and consists of  $$\{x,x^2 \}$$ that is those functions which fix $\vF_2$ and act as $ C_{2}$ on  $\vF_4\setminus\vF_2$. This is also $H_2^2$.
%\end{example}

\section{Conclusions}

The set of canonical subfield preserving polynomials has been studied and a monoid structure has been provided via combinatorial arguments  (Section \ref{CU} and \ref{SS}) . The density of this set has been addressed yielding curious results at least for the case of prime degree extension  (Section \ref{AD}). A simple example has also been given (Section \ref{EE}).

\end{document}